%% file: squares-lower-bounds.tex
\documentclass{amsart}
\usepackage{tikz, rigid-figures}
\usepackage{thmtools,thm-restate}
\usepackage{ifthen}
\usepackage{scalefnt}
\usepackage[latin1]{inputenc}
\usepackage[T1]{fontenc}
\usepackage{amsmath,amssymb,amsthm,bm}
\usepackage{graphicx}
\usepackage{hyperref}
\usepackage{upgreek}
\usepackage{cleveref}

\usepackage[marginpar]{todo}

\usepackage{booktabs,siunitx}
\usepackage{scalerel}

\declaretheorem[numberwithin=section]{theorem}

\newcommand{\abs}[1]{\left\lvert #1 \right\rvert}
\newcommand{\co}{\colon\thinspace}

\newtheorem{lemma}[theorem]{Lemma}

\newtheorem{conjecture}[theorem]{Conjecture}

\newtheorem{proposition}[theorem]{Proposition}
\theoremstyle{definition}

\begin{document}
\date{\today}
\author[H. Alpert]{Hannah Alpert}
\address{Auburn University, 221 Parker Hall, Auburn, AL 36849}
\email{hcalpert@auburn.edu}
\author[M. Kahle]{Matthew Kahle}
\address{The Ohio State University, 231 West 18th Ave, Columbus, OH 43210}
\email{mkahle@math.osu.edu}
\thanks{MK gratefully acknowledges the support of NSF awards \#1839358 and \#2005630.}
\author[R. MacPherson]{Robert MacPherson}
\address{Institute for Advanced Study, 1 Einstein Drive, Princeton, NJ 08540}
\email{rdm@math.ias.edu}

\title[Asymptotic Betti numbers for hard squares]{Asymptotic Betti numbers for hard squares in the homological liquid regime}
\maketitle

\begin{abstract}
We study configuration spaces $C(n; p, q)$ of $n$ ordered unit squares in a $p$ by $q$ rectangle.  Our goal is to estimate the Betti numbers for large $n$, $j$, $p$, and $q$.  We consider sequences of area-normalized coordinates, where $(\frac{n}{pq}, \frac{j}{pq})$ converges as $n$, $j$, $p$, and $q$ approach infinity.  For every sequence that converges to a point in the ``feasible region'' in the $(x,y)$-plane identified in \cite{ABKMS20}, we show that the factorial growth rate of the Betti numbers is the same as the factorial growth rate of $n!$. This implies that (1) the Betti numbers are vastly larger than for the configuration space of $n$ ordered points in the plane, which have the factorial growth rate of $j!$, and (2) every point in the feasible region is eventually in the homological liquid regime. 
\end{abstract}

\section{Introduction}\label{sec:intro}

A typical phase diagram might look something like Figure \ref{fig:phase}. At high enough pressure, nearly any material undergoes a liquid--solid phase transition. In the physics literature, it has been suggested for decades that because these phenomena are so ubiquitous, they should have a general explanation \cite{Uhlenbeck}. One might imagine that at high pressure, attractive forces between molecules should not matter. Kirkwood posed the problem of whether a continuous system of particles with purely repulsive hard-core potential would exhibit a liquid--solid phase transition. Such phase transitions have been observed experimentally, but a theoretical description of the Kirkwood transition has remained elusive. 

\begin{figure}
    \centering
    \includegraphics[width=4in]{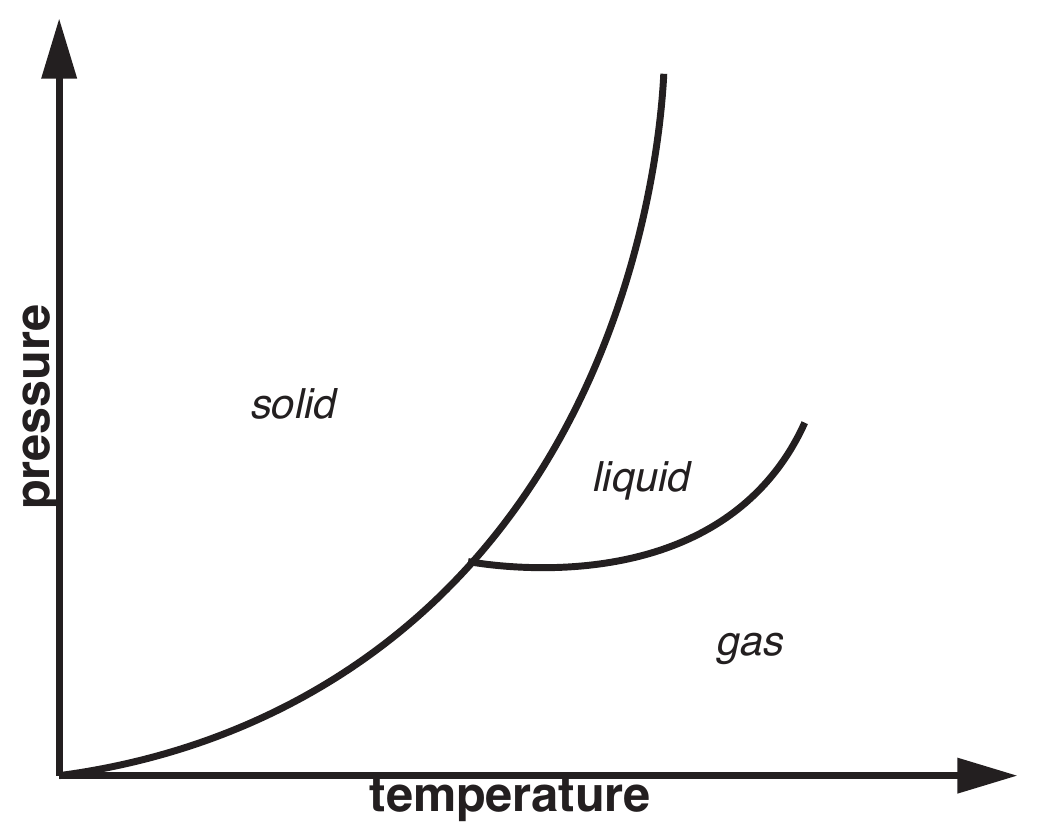}
    \caption{A typical phase diagram}
    \label{fig:phase}
\end{figure}

Motivated by such considerations, Diaconis suggested studying the topology of the configuration space of hard disks in a box in \cite{Diaconis09}.  
\begin{quote}
We know very, very little about the topology of the set $X(n, \epsilon)$ of configurations:
for fixed $n$, what are useful bounds on $\epsilon$ for the space to be connected? What are the Betti numbers? Of course, for $\epsilon$ ``small'' this set is connected but very little else is known.
\end{quote}
Motivated in part by these suggestions, in \cite{AKM19} we studied hard disks in an infinite strip. In that article, we obtained estimates for the rate of growth of the Betti numbers, and also introduced definitions of homological solid, liquid, and gas regimes. These definitions are flexible, and make sense for a wide variety of hard core models.

For a point in parameter space---that is, we specify the shapes of the particles and their container, as well as the degree of homology---we say that the point is in
\begin{itemize}
    \item the homological solid regime if homology is trivial,
    \item the homological gas regime if the inclusion map into the ambient configuration space of points induces an isomorphism on homology, and
    \item the homological liquid regime otherwise.
\end{itemize}

The homological liquid regime is the most interesting regime topologically. In \cite{AKM19}, it is shown that the Betti numbers grow exponentially fast with the number of particles. This is in sharp contrast to the homological gas regime, where they only grow polynomially fast.

In this article, we study the configuration space of $n$ hard squares in a $p \times q$ rectangle, denoted $C(n;p,q)$. A precise definition will appear below. These configuration spaces were studied topologically earlier in \cite{ABKMS20} and \cite{Plachta21}. Parallel hard squares have also been studied in the physics literature \cite{belli2012free, gonzalez2017dynamical,hoover2009single}

In \cite{ABKMS20}, we were mostly interested in understanding for which $n,j,p,q$ homology $H_j [C(n;p,q) ]$ can be nontrivial. We considered area-normalized coordinates $(x,y)$,
where $x=n/pq$ and $y=j/pq$. For each $k \geq 1$ we constructed a nontrivial element of $H_j[C(n; p, q)]$, with $p = q = 2k$, $n = 3k$, and $j = 3k-2$, so in area-normalized coordinates $(x,y) = (\frac{3}{4k}, \frac{3}{4k} - \frac{1}{2k^2})$.  We defined $S$ to be the set of all such points in the $(x,y)$-plane; that is, if we let $s_k$ denote $(\frac{3}{4k}, \frac{3}{4k} - \frac{1}{2k^2})$, then $S = \{s_k\}_{k \geq 1}$.  We also showed that if we let $S_+ := S \cup \{(0, 0), (1, 0)\}$, and if $\mathcal{R}$ is the convex hull of $S_+$, then for every point with rational coordinates $(x, y)$ in $\mathcal{R}$, there exist $j, n, p, q$ with $(\frac{n}{pq}, \frac{j}{pq}) = (x, y)$ such that $H_j[C(n; p, q)]$ is nontrivial.  Because of this, we refer to $\mathcal{R}$ as the ``feasible region'' in the area-normalized plane.

\input{feasible.tex}

As in~\cite{ABKMS20}, we define the configuration space $C(n; p, q)$ of $n$ unit squares in the $p$ by $q$ rectangle $\left[\frac{1}{2}, p - \frac{1}{2}\right] \times \left[\frac{1}{2}, q- \frac{1}{2}\right]$ in $\mathbb{R}^2$ to be the set of all points $(x_1, y_1, \ldots, x_n, y_n) \in \mathbb{R}^{2n}$ such that
\begin{itemize}
\item $1 \leq x_k \leq p$ and $1 \leq y_k \leq q$ for all $1 \leq k \leq n$, and
\item $\abs{x_k - x_\ell} \geq 1$ or $\abs{y_k - y_\ell} \geq 1$ for all $1 \leq k < \ell \leq n$.
\end{itemize}
In other words, these are disks in a box, but with the $L^{\infty}$ norm. We refer to the sliding unit squares as ``pieces'', and refer to the unit squares centered on integer lattice points of the plane as ``grid squares''.  The coordinates $(x_k, y_k)$ specify the center of the piece labeled $k$.

We are interested in the homology groups $H_j[C(n; p, q)]$.  Specifically, for which area-normalized coordinates $(\frac{n}{pq}, \frac{j}{pq})$ can $H_j[C(n; p, q)]$ possibly be nonzero?  And, given a sequence of $j$, $n$, $p$, and $q$ such that the area-normalized coordinates $(\frac{n}{pq}, \frac{j}{pq})$ converge to a particular value, what can we say about the growth of the dimension of $H_j[C(n; p, q)]$?  (When we refer to the dimension of a homology group, we are taking homology with coefficients in $\mathbb{Q}$ or any field.  In the case of $C(n; p, q)$ we do not know whether the homology with $\mathbb{Z}$ coefficients may even be free abelian, in which case all notions of dimension would be the same.)

To describe the growth of the dimension of $H_j[C(n; p, q)]$ as the area-normalized coordinates $(\frac{n}{pq}, \frac{j}{pq})$ converge to a given point, we use the terminology of factorial growth rate, defined 
as follows.  Let $f$ be a positive function on the natural numbers.  We say that $f$ has the \textit{\textbf{factorial growth rate}} of $(an)!$ if 
\[\lim_{n \rightarrow \infty} \frac{\ln f(n)}{n \ln n} = a.\]
Similarly, we say that the factorial growth rate is at least $(an)!$ if the corresponding $\liminf$ is at least $a$, and that the factorial growth rate is at most $(an)!$ if the corresponding $\limsup$ is at most $a$.  To verify that $(an)!$ does indeed have the factorial growth rate of $(an)!$, we check $\frac{\ln ((an)!)}{n\ln n} \rightarrow a$, which is true because of Stirling's approximation,
\[\lim_{n \rightarrow \infty} \frac{\sqrt{2\pi n} (\frac{n}{e})^n}{n!} = 1.\]

The purpose of this paper is to prove the following theorem.

\begin{theorem}\label{thm:strong-lower}
Let $(x, y)$ be any point in the interior of $\mathcal{R}$.   Let $n_i, j_i, p_i, q_i$ be sequences such that $p_i \leq q_i$, $p_i \rightarrow \infty$, $\frac{n_i}{p_iq_i} \rightarrow x$, and $\frac{j_i}{p_iq_i} \rightarrow y$.  Then $\dim H_{j_i}[C(n_i; p_i, q_i)]$ has the factorial growth rate of $n_i!$.
\end{theorem}

The main task is to prove the lower bound, that the factorial growth rate is at least $n_i!$, because the upper bound is an easy consequence of our theorem from~\cite{ABKMS20} showing that $C(n; p, q)$ is homotopy equivalent to a certain cubical complex. We can estimate the number of cells in this complex for an upper bound on the Betti numbers.

\begin{lemma}\label{lem:upper}
Let $n_i, j_i, p_i, q_i$ be sequences such that $n_i \rightarrow \infty$ and $p_iq_i \leq M\cdot n_i$ for some constant $M$.  Then $\dim H_{j_i}[C(n_i; p_i, q_i)]$ has the factorial growth rate of at most $n_i!$.
\end{lemma}

\begin{proof}
In~\cite{ABKMS20} we show that $C(n_i; p_i, q_i)$ is homotopy equivalent to a cubical complex $X(n_i; p_i, q_i)$, which is a subcomplex of the following cubical complex $G(n_i; p_i, q_i)$.  Let $[\frac{1}{2}, \ell + \frac{1}{2}]$ denote the complex consisting of an interval subdivided into $\ell$ $1$-cells and $\ell + 1$ $0$-cells.  Then we define 
\[G(n_i; p_i, q_i) = \left[\frac{1}{2}, p_i + \frac{1}{2}\right]^{n_i}\times\left[\frac{1}{2}, q_i + \frac{1}{2}\right]^{n_i}.\]
Thus, $\dim H_{j_i}[C(n_i; p_i, q_i)]$ is at most the number of $j_i$-cells in $G(n_i; p_i, q_i)$.

To estimate the number of $j_i$-cells in $G(n_i; p_i, q_i)$, there are $\binom{2n_i}{j_i} \leq 2^{2n_i}$ ways to select which of the coordinates are fixed and which are varying.  Then, for each of the first $n_i$ coordinates, there are $p_i + 1$ choices for its value if it is fixed, and $p_i$ choices for its interval of values if it is varying; for the remaining $n_i$ coordinates, there are $q_i + 1$ choices if fixed, and $q_i$ if varying.  Thus, the total number of $j_i$-cells in $G(n_i; p_i, q_i)$ is bounded by $2^{2n_i}(p_i + 1)^{n_i}(q_i + 1)^{n_i}$,
which for sufficiently large $i$ is at most $2^{2n_i}((M + 1)n_i)^{n_i} = (4(M+1)n_i)^{n_i}$, and this upper bound has the factorial growth rate of $n_i!$.
\end{proof}

In Theorem~\ref{thm:strong-lower} the growth rate of $\dim H_{j_i}[C(n_i; p_i, q_i)]$ is much faster than it would be if $p_i$ and $q_i$ were larger, for instance if they were as large as $n_i$.  The configuration space $C(n; n, n)$ is homotopy equivalent to the configuration space of $n$ ordered distinct points in the plane~\cite{AKM19}.  We denote this space by $C(n)$ and show in the next proposition that when $j$ and $n$ grow in constant ratio, then $\dim H_j[C(n)]$ only grows like $j!$, not like $n!$.

\begin{proposition}
Let $n_i$ and $j_i$ be sequences such that $j_i \rightarrow \infty$ and $j_i \leq \alpha \cdot n_i$ for some constant $\alpha < 1$.  Then $\dim H_{j_i}[C(n_i)]$ has the factorial growth rate of $j_i!$.
\end{proposition}

\begin{proof}
Arnold~\cite{Arnold14} showed that $H_*[C(n)]$ and $H^*[C(n)]$ are free abelian groups and that the Poincar\'e polynomial of $C(n)$ is given by
\[\beta_0 + \beta_1 t + \beta_2 t^2 + \cdots + \beta_{n-1}t^{n-1} = 1(1+t)(1+2t)\cdots(1+(n-1)t),\]
where $\beta_j$ denotes $\dim H_j[C(n)]$.  Thus we have
\[\dim H_j[C(n)] = \sum_{\substack{S \subseteq \{1, \ldots, n-1\}\\ \abs{S} = j}} \prod_{i \in S}i.\]
This sum is bounded below by $j!$, which is one of the summands.  It is bounded above by the number of summands times the greatest summand, which is
\[\binom{n-1}{j} \cdot (n-1)(n-2)\cdots(n-j) \leq 2^n \cdot n^j.\]
Thus we have
\[\frac{\ln(\dim H_{j_i}[C(n_i)])}{j_i \ln j_i} \leq \frac{n_i \ln 2 + j_i \ln n_i}{j_i\ln j_i} \leq \frac{\alpha j_i \ln 2 + j_i(\ln \alpha + \ln j_i)}{j_i \ln j_i},\]
which approaches $1$ as $j_i \rightarrow \infty$.
\end{proof}

In Section~\ref{sec:warmup} we prove an easier special case of the main theorem, as a way of introducing the concepts that go into the main proof.  In Section~\ref{sec:whytile} we define ``rigid cycle tile'', and show that the main theorem can be reduced to the task of constructing a large rigid cycle tile made out of smaller rigid cycle tiles.  In Section~\ref{sec:maketile} we construct the smaller rigid cycle tiles that are the building blocks of our main proof, and in Section~\ref{sec:proofs} we prove the main theorem using two lemmas, showing that we can find a suitable collection of smaller tiles and assemble them together.  Section~\ref{sec:conclusion} gives some questions and directions for future work.

\section{Warm-up theorem}\label{sec:warmup}

In this section, we give an informal proof of the following special case of our main theorem, in order to give an idea of the reasons for the main theorem.

\begin{theorem}\label{thm:warmup}
Let $p$ and $q$ be even numbers, let $j = \frac{pq}{4}$, and let $n = 2j$.  Then we have
\[\dim H_j[C(n; p, q)] \geq \frac{n!}{2^j}.\]
\end{theorem}

\begin{figure}
    \centering
    \figwarmup
    \caption{Each way to assign the labels $1$ through $n$ in pairs to the $2$ by $2$ squares results in another cycle, and these cycles are linearly independent.  To show this, we pair the cycles with a corresponding set of cocycles.}
    \label{fig:warmup}
\end{figure}

\subsection{Maps to and from the $j$-torus}
We construct a class in $H_j[C(n; p, q)]$ according to Figure~\ref{fig:warmup} in the following way.  The $p$ by $q$ rectangle is divided into $2$ by $2$ squares, each containing one pair of pieces.  Each pair of pieces can orbit each other within their $2$ by $2$ square, creating one loop's worth of configurations.  The $j$ pairs can move independently, so the resulting motion can be parametrized by a $j$-dimensional torus $T^j$.  In other words, we have described an embedding $\phi\co T^j \rightarrow C(n; p, q)$.  Because $T^j$ is an orientable closed manifold, it has a fundamental homology class, and $\phi$ pushes forward this class to give $\phi_*[T^j] \in H_j[C(n; p, q)]$.

To show that this homology class is nontrivial, we exhibit a cohomology class in $H^j[C(n; p, q)]$ that pairs nontrivially with it.  We consider the vector from the center of piece $1$ to the center of piece $2$, the vector from $3$ to $4$, and so on, for a total of $j$ nonzero vectors.  These vectors are drawn as arrows in Figure~\ref{fig:warmup}.  We normalize these vectors to unit length to obtain a map $\psi \co C(n; p, q) \rightarrow T^j$.  Pulling back the fundamental cohomology class of $T^j$ gives the cohomology class $\psi^*[T^j] \in H^j[C(n; p, q)]$.  The composition $\psi \circ \phi \co T^j \rightarrow T^j$ is the identity map and in particular has degree $1$, so the pairing of $\phi_*[T^j]$ with $\psi^*[T^j]$ evaluates to $1$, proving that both classes are nontrivial.

\subsection{Permuting the labels}
Each permutation $\sigma \in S_n$ acts on $C(n; p, q)$ by permuting the labels of the pieces.  Thus, the maps $\phi$ and $\psi$ give rise to maps $\sigma \circ \phi$ and $\psi \circ \sigma^{-1}$, which correspond to homology and cohomology classes of their own.  However, not all of these classes are different.  Swapping labels $1$ and $2$, for instance, results in reparametrizing their circle coordinate in $\phi$ by half a turn, giving a map homotopic to $\phi$.  Thus, we restrict our attention to the permutations $\sigma \in S^n$ that preserve the order of labels on each of our $j$ pairs; that is, $\sigma(1) < \sigma(2)$, $\sigma(3) < \sigma(4)$, and so on.  The number of such permutations is $\frac{n!}{2^j}$.  We claim that the resulting homology classes $(\sigma \circ \phi)^*[T^j]$ are linearly independent.

To prove the linear independence, for each homology class in our collection we would like to exhibit a corresponding cohomology class, in such a way that the pairing of the corresponding classes evaluates to $1$ and the pairing of non-corresponding classes evaluates to $0$.  Unfortunately, the cohomology classes $(\psi \circ \sigma^{-1})^*[T^j]$ do not have this property.  For instance, if $\sigma$ is the permutation that swaps labels $1$ and $3$ and swaps labels $2$ and $4$, then $\psi \circ \sigma^{-1}$ differs from $\psi$ only by the reparametrization of $T^j$ that swaps the first two coordinates.  This implies that $(\psi \circ \sigma^{-1})^*[T^j]$ pairs nontrivially with $\phi_*[T^j]$.  Thus, we need a more refined method of constructing the corresponding cohomology class for each homology class in our collection.
 
\subsection{Subdividing cocycles}
Another way to think of the cohomology class $\psi^*[T^j]$ is in terms of intersection number.  The map $\psi$ can be defined on the configuration space of $n$ ordered distinct points in the plane, which we denote by $C(n)$.  Abusing notation, we use $\psi$ to refer both to the map on $C(n; p, q)$ and to the map on $C(n)$.  We note that $C(n; p, q)$ is a subspace of $C(n)$ and that $C(n)$ is an open manifold.  Selecting any point in $T^j$, we may consider its preimage in $C(n)$ under $\psi$, to get a submanifold $V$ of codimension $j$.  Given any $j$-dimensional cycle in $C(n)$, we may pair it with $\psi^*[T^j]$ by taking its intersection number with $V$.

In the case where the intersection of $V$ with $C(n; p, q)$ is disconnected, for any $j$-dimensional cycle in $C(n; p, q)$, its intersection number with $V$ is the sum of the contributions from the various connected components.  Thus, the pullback of the cohomology class to $C(n; p, q)$ is a sum of classes, one for each connected component.  (If this claim seems insufficiently justified, do not worry!  The present section is meant to be informal, and there is a more rigorous treatment in Section~\ref{sec:whytile}.)  We will use this observation to find a class that, like $\psi^*[T^j]$, gives $1$ when paired with $\phi^*[T^j]$, but unlike $\psi^*[T^j]$, gives $0$ when paired with the other homology classes in our collection.  This class corresponds to one component of $V \cap C(n; p, q)$.

Specifically, the point we select in $T^j$ is the one corresponding to the configuration in Figure~\ref{fig:warmup}, with coordinates that alternate between $\frac{\pi}{4}$ and $-\frac{\pi}{4}$.  The component of $V$ we select is the one containing this configuration in the figure.  The important property of this choice is that this component of $V \cap C(n; p, q)$ contains only this one configuration; that is, there is no way to perturb the pieces while maintaining the slopes of the $j$ arrows.  Moreover, we observe that none of the images of the other tori $\sigma \circ \phi$ contain this configuration.  This implies that the pairing of our new cohomology class with each of the other homology classes in our collection is $0$, while its pairing with $\phi^*[T^j]$ is $1$.

Repeating this process, we obtain one cohomology class in $H^j[C(n; p, q)]$ for each of our $\frac{n!}{2^j}$ homology classes.  These classes pair perfectly, so the homology classes are linearly independent (and so are the cohomology classes).

\section{Permutation families of independent cycles}\label{sec:whytile}

To prove the main theorem, we construct pictures similar to Figure~\ref{fig:warmup}.   The purpose of this section is to make precise what comprises such a picture, and to show that drawing such a picture implies the desired lower bound on homology.

We define a \textit{\textbf{rigid stress tile}} for some $n, j, p, q$ to be a drawing of $n$ shaded grid squares in a $p$ by $q$ rectangle, with bars connecting $j$ pairs of diagonally adjacent shaded grid squares.  We require, informally, that there is no way to perturb the $n$ pieces while preserving the slopes of the bars.  More formally, we label the shaded squares $1$ through $n$, and consider the map $\psi \co C(n; p, q) \rightarrow T^j$ that records, for each bar between two squares, the unit vector from the center of the lesser-numbered square toward the center of the greater-numbered square.  To be a rigid stress tile, we require the configuration in our drawing to be an isolated point in its fiber under this map $\psi$.

We define a \textit{\textbf{rigid cycle tile}} to be a rigid stress tile together with a map $\phi \co T^j \rightarrow C(n; p, q)$ such that the map $\psi \circ \phi \co T^j \rightarrow T^j$ is a homeomorphism.  If the rigid cycle tile is called $T$, then the corresponding values of $n, j, p, q$ are denoted by $n(T), j(T), p(T), q(T)$.  We let $a(T)$ denote the area $p(T)q(T)$, and we let $v(T)$ denote the vector $(n(T), j(T), a(T))$.  We abbreviate the area-normalized coordinates $\left(\frac{n(T)}{a(T)}, \frac{j(T)}{a(T)}\right)$ by $(\frac{n}{a}, \frac{j}{a})(T)$.  In Theorem~\ref{thm:warmup} we constructed a large rigid cycle tile from many $2$ by $2$ tiles, and wanted to count the permutations that preserved the order within each of the smaller tiles.  For the general case, we refer to these smaller tiles as the \textit{\textbf{clusters}} of the rigid cycle tile, defined as follows.  We generate an equivalence relation by specifying that two of the $n(T)$ pieces are equivalent if their movements under $\phi$ intersect; that is, there are two configurations in the image of $\phi$ such that the first piece in the first configuration and the second piece in the second configuration have points of the rectangle in common.  The equivalence classes of the resulting equivalence relation are the clusters.  If two pieces are in different clusters, then under $\phi$ they stay in disjoint parts of the rectangle.

The following lemma shows that constructing a rigid cycle tile is enough to give a lower bound on homology.

\begin{lemma}\label{lem:permute}
Let $T$ be a rigid cycle tile, with clusters of size $n_\ell$ with $\sum_\ell n_\ell = n(T)$.  Then we have
\[\dim H_{j(T)}[C(n(T); p(T), q(T))] \geq \frac{n(T)!}{\prod_\ell n_\ell!}.\]
\end{lemma}

\begin{proof}
Let $n=n(T)$, $j=j(T)$, $p=p(T)$, and $q=q(T)$.  Let $c \in C(n; p, q)$ be the configuration designated by $T$.  The group $S_n$ acts on $C(n; p, q)$ by permuting the labels of the $n$ pieces.  Let $E$ be the subset of $S_n$ consisting of all permutations that are order-preserving on each cluster of $c$.  We have
\[\abs{E} = \frac{n!}{\prod_\ell n_\ell!}.\]
For each permutation $\sigma \in E$, the map $\sigma \circ \phi \co T^j \rightarrow C(n; p, q)$ gives a homology class $[\sigma \circ \phi] = (\sigma \circ \phi)_*[T^j] \in H_j[C(n; p, q)]$.  To prove the lemma, it suffices to show that these homology classes are linearly independent.  In what follows, we define a linear functional $f \co H_j[C(n; p, q)] \rightarrow \mathbb{Z}$ such that $f([\phi])$ is nonzero, but $f([\sigma \circ \phi]) = 0$ for all $\sigma \in E$.  Then the functionals $f \circ \sigma^{-1}$ give a dual basis to our set of homology classes $[\sigma \circ \phi]$, which then must be linearly independent.

Let $x = \psi(c) \in T^j$.  For a sufficiently small $\varepsilon$ that we choose later, we take the ball $B_{\varepsilon}(x)$ around $x$ in $T^j$, let $U$ be the component of $\psi^{-1}(B_\varepsilon(x))$ in $C(n; p, q)$ that contains $c$, and let $V$ be the remainder of $\psi^{-1}(B_\varepsilon(x))$.  Then, letting $C = C(n; p, q)$ for brevity, we choose the functional $f$ to be the composite map
\[H_j(C) \rightarrow H_j(C, C\setminus U) \cong H_j(C\setminus V, C\setminus(U\cup V)) \rightarrow H_j(T^j, T^j \setminus B_\varepsilon(x)) \cong \mathbb{Z},\]
where the isomorphism in the second map is by excision, and the third map is induced by $\psi$.

We choose $\varepsilon$ as follows.  The definition of rigid stress tile implies that $\psi^{-1}(x) \setminus c$ is a closed, hence compact, subset of $C(n; p, q)$.  Also, for each $\sigma \in E$ other than the identity, the image of $\sigma \circ \phi$ in $C(n; p, q)$ does not contain $c$, because the only element of $E$ that keeps every piece within its original cluster is the identity permutation.  Thus the set
\[K = (\psi^{-1}(x) \setminus c) \cup \bigcup_{1 \neq \sigma \in E} \sigma\circ\phi(T^j)\]
has some positive distance $\delta$ from $c$ in $C(n; p, q)$, if we use the subspace metric inherited from $\mathbb{R}^{2n}$.  We would like to select $\varepsilon$ in such a way that the resulting neighborhood $U$ of $c$ is disjoint from $K$.  To do this, we observe that the set $K' = C(n; p, q) \setminus B_{\delta/3}(\psi^{-1}(x))$ is compact, and $\psi(K')$ does not contain $x$, so we may choose $\varepsilon$ small enough that $B_{\varepsilon}(x)$ is disjoint from $\psi(K')$.  This then implies that if we let $U$ be the connected component of $\psi^{-1}(B_{\varepsilon}(x))$ that contains $c$, then $U$ does not contain any other points of $\psi^{-1}(x)$ or of the images of $\sigma \circ \phi$, for non-identity permutations $\sigma\in E$.

To show that $f([\phi]) \neq 0$, we use the fact that $\psi \circ \phi \co T^j \rightarrow T^j$ is a homeomorphism to claim that the image of $\phi$ in $C(n; p, q)$ is disjoint from $V$.  This is because the intersection of $U$ with the image of $\phi$ is already mapped by $\psi$ onto $B_\varepsilon(x)$, so no other points in the image of $\phi$ can map into $B_\varepsilon(x)$.  Thus, applying $f$ to $[\phi]$ is the same as taking the relative class of $\psi \circ \phi$ in $H_j(T^j, T^j \setminus B_{\varepsilon}(x))$, which is $\pm 1 \in \mathbb{Z}$.

To show that $f([\sigma \circ \phi]) = 0$ for all non-identity permutations $\sigma \in E$, our construction of $U$ guarantees that the image of $\sigma \circ \phi$ is disjoint from $U$.  Thus, in the composition that gives $f$, the first map into $H_j(C, C \setminus U)$ sends $[\sigma \circ \phi]$ to zero.
\end{proof}

\section{Building the tiles}\label{sec:maketile}

In this section, we describe the rigid cycle tiles needed for the proof of the main theorem, Theorem~\ref{thm:strong-lower}.

\subsection{Single-cluster tiles} 
We specify cycles in $C(n; p, q)$ as in~\cite{ABKMS20}.  Roughly, given a larger box in the plane containing two or more smaller boxes, we may be able to move the smaller boxes to make a piecewise linear loop in the configuration space of the smaller boxes in the larger box.  For instance, a box of side length $k$ and up to $k+1$ boxes of side length $1$ can orbit each other in a box of side length $k+1$, and two or three boxes of side length $k$ can orbit each other in a box of side length $2k$.  The smaller boxes may, in turn, enclose even smaller boxes orbiting each other, making a nested pattern of boxes ending with our pieces of side length $1$ at the smallest and innermost level.  The degree of the resulting cycle is the number of circle parameters, which is the number of boxes we have drawn that enclose either pieces or smaller boxes.

\begin{figure}
    \centering
    \figsinglecluster
    \caption{The rigid cycle tiles $I_k$, $\Delta_k$, and $Y_{2k}$ have one cluster each, and their recursively-constructed cycles can be represented by drawings of nested boxes.}
    \label{fig:singlecluster}
\end{figure}

Our construction relies on the following types of single-cluster rigid cycle tiles, depicted in Figure~\ref{fig:singlecluster}.
\begin{itemize}
    \item The ``diagonal $k$-tile'' $I_k$ is the $k$ by $k$ rigid cycle tile consisting of $k$ pieces on a diagonal, forming a $(k-1)$-cycle.  This cycle is also constructed in Lemma~6.2 of~\cite{ABKMS20}.
    \item The ``triangular $k$-tile'' $\Delta_k$ is the rigid cycle tile formed by taking $I_k$ and packing one side of the diagonal with extra pieces.
    \item The ``extremal $2k$-tile'' $Y_{2k}$ is a rigid cycle tile that consists of three copies of $I_k$ orbiting each other in a $2k$ by $2k$ square.  The corresponding points $(\frac{n}{a}, \frac{j}{a})(Y_{2k})$ are equal to $s_k$, the vertices of the feasible region.  This cycle is also constructed in Lemma~6.3 of~\cite{ABKMS20}.
\end{itemize}

\subsection{Multi-cluster fenced tiles}
There are many ways to arrange rigid cycle tiles side by side to form larger tiles.  Not every way of arranging them maintains the property that there is no way to perturb the pieces while preserving the slopes of the bars, but if this property holds, then the resulting large tile is a rigid cycle tile, and its degree is the sum of degrees of the smaller cycles that comprise it.

We say that a rigid cycle tile is \textit{\textbf{fenced}} if all of the squares along the boundary of the rectangle are shaded.  Fenced tiles are particularly useful, because if a rectangle is partitioned into fenced rigid cycle tiles, then it is automatically a (fenced) rigid cycle tile.

\begin{figure}
    \centering
    \figquilt
    \caption{Each fenced quilt tile $Q_{k, w, h}$ is formed by surrounding a grid of copies of $X_{4k}$ by a border of triangular tiles $\Delta_k$, and with solid $1$-tiles in the corners.  The tile with $k = 3$, $w = 3$, and $h = 2$ is pictured.}
    \label{fig:quilt}
\end{figure}

We use the following constructions to assemble our single-cluster rigid cycle tiles into bigger tiles, and then into fenced tiles.
\begin{itemize}
    \item The ``symmetric extremal $4k$-tile'' $X_{4k}$ consists of four copies of $Y_{2k}$, arranged symmetrically to form a rigid cycle tile with the four empty quadrants together at the center.
    \item The ``solid $1$-tile'' $I_1$ consists of a single shaded square forming a $0$-cycle.
    \item The ``fenced quilt tile'' $Q_{k, w, h}$ consists of $wh$ copies of $X_{4k}$, arranged in a rectangle $w$ copies wide and $h$ copies high, with a border of triangular $k$-tiles $\Delta_k$, with solid $1$-tiles to fill in the corners, as in Figure~\ref{fig:quilt}.
    \item The ``fenced one-hole $2k$-tile'' $O_{2k}$ consists of four copies of $\Delta_k$ arranged symmetrically with the four empty quadrants together at the center, as in Figure~\ref{fig:ob}.
    \item The ``fenced two-hole tile'' $B_6$ consists of one copy of $I_2$, six copies of $\Delta_2$, and solid $1$-tiles to fill in the corners, as in Figure~\ref{fig:ob}.
\end{itemize}

\begin{figure}
    \centering
    \figob
    \caption{The small fenced tiles $O_{2k}$ and $B_6$ are possible ways to combine extra triangular tiles, or copies of the diagonal tile $I_2$.}
    \label{fig:ob}
\end{figure}

\section{Main proof}\label{sec:proofs}

In this section, Lemma~\ref{lem:sum} shows that we can figure out the right collection of tiles, and Lemma~\ref{lem:pack} shows that we can fit those tiles in the big square.  Using these two lemmas, we can finish the proof of our main theorem, Theorem~\ref{thm:strong-lower}.

In both lemmas, our approach is to divide the feasible region $\mathcal{R}$ into a sequence of triangles.  In $S_+$, we observe that the vertex $(0, 0)$ is the limit of the vertices $s_k = (\frac{3}{4k}, \frac{3}{4k} - \frac{1}{2k^2})$.  Thus, omitting $(0, 0)$ from $S_+$ does not change the interior of the convex hull.  We can divide the interior into triangles by connecting each vertex $s_k$ to the vertex $(1, 0)$.  The triangles are formed by $(1, 0)$ and two consecutive vertices $s_k$ and $s_{k+1}$.  Every point $(x, y)$ in the interior of $\mathcal{R}$ is either in the interior of one of these triangles, or on the interior of a line segment from some $s_k$ to $(1, 0)$.  Thus, given $(x, y)$ we can find the corresponding value of $k$ and focus our attention only on the triangle with vertices $(1, 0)$, $s_k$, and $s_{k+1}$.

The first of the two lemmas in this section shows that if $(\frac{n_i}{p_iq_i}, \frac{j_i}{p_iq_i})$ is sufficiently close to $(x, y)$, then we can find a collection of rigid cycle tiles such that the totals of their contributions to number of pieces, degree, and area are $n_i$, $j_i$, and $p_iq_i$.

\begin{lemma}\label{lem:sum}
Suppose that $p_i \leq q_i$, $p_i \rightarrow \infty$, and the sequence $\left(\frac{n_i}{p_iq_i}, \frac{j_i}{p_iq_i}\right)$ converges to a point (x, y) in the interior of the convex hull of $s_k = (\frac{n}{a}, \frac{j}{a})(Y_{2k})$, $s_{k+1} = (\frac{n}{a}, \frac{j}{a})(Y_{2(k+1)})$, and $(1, 0) = (\frac{n}{a}, \frac{j}{a})(I_1)$, or on the interior of the segment from $s_k$ to $(1, 0)$.  Then there is some $\varepsilon > 0$ depending on $(x, y)$ such that for all sufficiently large $i$, the vector $(n_i, j_i, p_iq_i)$ can be written as the sum of the vectors $v(T)$ as $T$ ranges over the following collection of rigid cycle tiles:
\begin{itemize}
    \item at most one large fenced quilt tile $Q_{k, w_1, h_1}$, with total width between $p_i-4k$ and $p_i$;
    \item at most one skinny fenced quilt tile $Q_{k, w_3, 1}$, with total width at most $p_i$;
    \item at most one large fenced quilt tile $Q_{k+1, w_2, h_2}$, with total width between $p_i-4(k+1)$ and $p_i$;
    \item at most one skinny fenced quilt tile $Q_{k+1, w_4, 1}$, with total width at most $p_i$;
    \item at most $2q_i$ of each of the fenced one-hole tiles $O_{2k}$ and $O_{2(k+1)}$;
    \item at most $100k^2$ of the fenced two-hole tiles $B_6$;
    \item at most $200k^2$ of the fenced one-hole tiles $O_4$;
    \item at most $3$ of the triangular tiles $\Delta_2$; and
    \item at least $\varepsilon \cdot p_iq_i$ of the solid $1$-tiles $I_1$.
\end{itemize}
\end{lemma}

\begin{proof}
To find $\varepsilon$, we have assumed that $(x, y)$ can be written as the convex combination $\lambda_1 s_k + \lambda_2s_{k+1} + \lambda_3(1, 0)$ with $\lambda_1 + \lambda_2 + \lambda_3 = 1$ and $\lambda_1, \lambda_3 > 0$.  We take $\varepsilon = \frac{1}{2}\lambda_3$.

We subtract $8q_i \cdot v(\Delta_k)$, $8q_i \cdot v(\Delta_{k+1})$, and $700k^2\cdot v(\Delta_2)$ from $(n_i, j_i, p_iq_i)$ to get some vector $(N, J, A)$.  If $i$ is sufficiently large, then $(\frac{N}{A}, \frac{J}{A})$ is still very close to $(x, y)$.  If $(x, y)$ is on the segment from $s_k$ to $(1, 0)$, then we can assume that $\left(\frac{N}{A}, \frac{J}{A}\right)$ is inside the triangle between $s_k$, $s_{k+1}$, and $(1, 0)$.  To see why, we observe that the line through $(1, 0)$ and $s_k$ corresponds to the set of points where the ratio between $a-n$ (that is, number of unshaded grid squares) and $j$ is equal to $k\cdot \frac{4k-3}{3k-2}$.  The point $s_{k+1}$ is on the side of the line with a greater such ratio, whereas the vectors $v(\Delta_k)$, $v(\Delta_{k+1})$, and $v(\Delta_2)$ have a lesser such ratio, so by subtracting them we move toward the side of the line that $s_{k+1}$ is on.

We can write $\left(\frac{N}{A}, \frac{J}{A}\right) = \lambda'_1 s_k + \lambda'_2 s_{k+1} + \lambda'_3(1, 0)$ with $\lambda'_1 + \lambda'_2 + \lambda'_3 = 1$.  (As $i$ becomes large, we know that $\lambda'_1, \lambda'_2, \lambda'_3$ approach $\lambda_1, \lambda_2, \lambda_3$.)  Because the areas of tiles $X_{4k}$ and $X_{4(k+1)}$ are $16k^2$ and $16(k+1)^2$, this implies
\[(N, J, A) = \frac{\lambda'_1\cdot A}{16k^2}\cdot v(X_{4k}) + \frac{\lambda'_2\cdot A}{16(k+1)^2}\cdot v(X_{4(k+1)}) + \lambda'_3\cdot A\cdot (1, 0, 1).\]
We set $L_1 = \left \lfloor \frac{\lambda'_1\cdot A}{16k^2}\right \rfloor$, $L_2 = \left\lfloor \frac{\lambda'_2\cdot A}{16(k+1)^2}\right\rfloor$, and $L_3 = \lfloor \lambda'_3\cdot A\rfloor$.  Then we define 
\[(N', J', A') = (N, J, A) - L_1 \cdot v(X_{4k}) - L_2 \cdot v(X_{4(k+1)}) - L_3 \cdot v(I_1).\] 
This remainder $(N', J', A')$ is a linear combination of $v(X_{4k})$, $v(X_{4(k+1)})$, and $v(I_1)$ with all coefficients between $0$ and $1$.  In particular, each of $N'$, $J'$, and $A'$ is less than $100k^2$.  Expanding $(N', J', A')$ in the basis for $\mathbb{R}^3$ given by $(3, 1, 4)$, $(2, 1, 4)$, and $(1, 0, 1)$, we have
\begin{align*}
(N',\ &J', A') \\
&= [2J' - (A'-N')] \cdot (3, 1, 4) + (A' - N' - J') \cdot (2, 1, 4) + (A' - 4J')\cdot (1, 0, 1)\\
&= [2J' - (A' - N')] \cdot v(\Delta_2) + (A' - N' - J') \cdot v(I_2) + (A' - 4J')\cdot v(I_1).
\end{align*}
Next we observe that the second and third coefficients are nonnegative integers.  This is because the inequalities $j \leq a-n$ and $4j \leq a$ are true of $v(X_{4k})$, $v(X_{4(k+1)})$, and $v(I_1)$, so they are also true of a positive linear combination of them.  The first coefficient $2J' - (A' - N')$ may be negative, but if so, its magnitude is at most $A' \leq 100k^2$; if it is positive, its magnitude is at most $J' \leq 100k^2$, since $A'-N'-J'$ is nonnegative.  Thus, we have
\[100k^2\cdot v(\Delta_2) + (N', J', A') = c_1\cdot v(\Delta_2) + c_2 \cdot v(I_2) + c_3\cdot v(I_1),\]
where $0 \leq c_1 \leq 200k^2$, $0 \leq c_2 \leq 100k^2$, and $0 \leq c_3$.

In this way, we have written $(n_i, j_i, p_iq_i)$ as a sum of $v(T)$, where $T$ ranges over the following collection of tiles:
\begin{itemize}
    \item $L_1$ copies of $X_{4k}$,
    \item $L_2$ copies of $X_{4(k+1)}$,
    \item $8q_i$ copies of $\Delta_k$,
    \item $8q_i$ copies of $\Delta_{k+1}$,
    \item at most $100k^2$ copies of $I_2$,
    \item between $600k^2$ and $800k^2$ copies of $\Delta_2$, and
    \item at least $L_3$ copies of the solid $1$-tile $I_1$.
\end{itemize}
What remains is to assemble these into the fenced tiles of the lemma statement.

We make fenced quilted tiles as follows.  Let $w_1$ be the largest integer such that $4k \cdot w_1 + 2k \leq p_i$; that is, a quilted tile that is $w_1$ copies of $X_{4k}$ wide can fit in the $p_i$ by $q_i$ board.  Similarly, let $w_2$ be the largest integer such that $4(k+1) \cdot w_2 + 2k \leq p_i$.  Let $h_1$ and $w_3$ be the quotient and remainder when $L_1$ is divided by $w_1$, and let $h_2$ and $w_4$ be the quotient and remainder when $L_2$ is divided by $w_2$.  Then we construct the fenced quilted tiles $Q_{k, w_1, h_1}$, $Q_{k, w_3, 1}$, $Q_{k+1, w_2, h_2}$, and $Q_{k+1, w_4, h_2}$ using the $L_1$ copies of $X_{4k}$, the $L_2$ copies of $X_{4(k+1)}$, an appropriate number of copies of $\Delta_k$ and $\Delta_{k+1}$, and $8k^2 + 8(k+1)^2$ of the solid $1$-tiles $I_1$.

Not all copies of $\Delta_k$ and $\Delta_{k+1}$ are used up in making the fenced quilted tiles.  However, the number of each began as a multiple of $4$, and the number of each that we used to make the quilted tiles is a multiple of $4$.  Thus, the remaining copies can be put together in sets of $4$ to create some number of fenced one-hole tiles $O_{2k}$ and $O_{2(k+1)}$, with the number of each bounded by $2q_i$.

We combine each copy of $I_2$ with six copies of $\Delta_2$ and eight solid $1$-tiles to form the fenced two-hole tile $B_6$.  Then, we put together the remaining copies of $\Delta_2$ in sets of $4$ to create at most $200k^2$ fenced one-hole tiles $O_4$, with a remainder of at most $3$ copies of $\Delta_2$.  The result is the collection of tiles in the lemma statement.

To estimate the number of solid $1$-tiles $I_1$ in the final collection, we start with at least $L_3$ of them, and then use a bounded number of them to build the fenced quilted tiles and the fenced two-hole tiles $B_6$.  Because $L_3/p_iq_i$ converges to $\lambda_3 = 2\varepsilon$ as $i$ goes to $\infty$, for large $i$ after subtracting a bounded number from $L_3$ there are still at least $\varepsilon \cdot p_iq_i$ solid $1$-tiles.
\end{proof}

The next lemma shows that the collection of tiles we have just constructed can actually fit in the $p_i$ by $q_i$ rectangle.  This relies on having sufficiently many solid $1$-tiles to fill in between the larger tiles.

\begin{lemma}\label{lem:pack}
Suppose that $p_i \leq q_i$, $p_i \rightarrow \infty$, and the sequence $\left(\frac{n_i}{p_iq_i}, \frac{j_i}{p_iq_i}\right)$ converges to a point (x, y) in the interior of the convex hull of $s_k$, $s_{k+1}$, and $(1, 0)$, or on the interior of the segment from $s_k$ to $(1, 0)$.  Then for all sufficiently large $i$, the collection of rigid cycle tiles from Lemma~\ref{lem:sum} can be arranged to form a $p_i$ by $q_i$ rigid cycle tile.
\end{lemma}

\begin{proof}
Our strategy is to combine the rigid cycle tiles in the collection into larger rigid cycle tiles, each a rectangle of width $p_i$.  If this collection of strips uses up all the tiles in the collection that are not solid $1$-tiles $I_1$, but uses up at most $M\cdot q_i$ solid $1$-tiles for some constant $M$ that depends on $k$ but not on $q_i$, then for sufficiently large $i$ there are solid $1$-tiles left over, because $\lim \frac{M \cdot q_i}{\varepsilon\cdot p_iq_i} = 0$.  This implies that the total area of our collection of strips is less than $p_iq_i$, so the total height is less than $q_i$.  By stacking the strips on top of each other, and putting in a strip of solid $1$-tiles to cover the remainder of the area, we can be sure that the strips fit in the $p_i$ by $q_i$ rectangle.

The large fenced quilt tiles $Q_{k, w_1, h_1}$ and $Q_{k+1, w_2, h_2}$ fill almost all of the width $p_i$.  We stack them vertically and need to add at most $(k+1) \cdot q_i$ solid $1$-tiles to fill the rest of the width.  The skinny fenced quilt tiles $Q_{k, w_3, 1}$ and $Q_{k+1, w_4, 1}$ may have small width but have bounded height, so when we stack them vertically, we need to add at most $6(k+1)\cdot p_i$ solid $1$-tiles to fill the rest of the width.

We can arrange the $2q_i$ fenced one-hole tiles $O_{2k}$ in rows, with up to $\left\lfloor \frac{p_i}{2k}\right\rfloor \geq \frac{p_i - 2k}{2k}$ copies of $O_{2k}$ in each row.  The total number of such rows that we need is at most $\frac{2q_i \cdot 2k}{p_i - 2k}$.  Similarly, we can arrange the $2q_i$ copies of $O_{2(k+1)}$ into at most $\frac{2q_i \cdot 2(k+1)}{p_i - 2(k+1)}$ rows.  The number of solid $1$-tiles needed to fill in around them is the difference between the area of the rectangle of width $p_i$ enclosing these rows, and the total area of the tiles.  This is at most
\[\frac{2q_i \cdot 2k}{p_i - 2k} \cdot 2k \cdot p_i - 2q_i(2k)^2 + \frac{2q_i \cdot 2(k+1)}{p_i - 2(k+1)} \cdot 2(k+1) \cdot p_i - 2q_i(2(k+1))^2\]
\[ = 2q_i \cdot \frac{(2k)^3}{p_i - 2k} + 2q_i \cdot \frac{(2(k+1))^3}{p_i - 2(k+1)},\]
which in particular is at most a constant multiple of $q_i$.

The bounded number of copies of $B_6$ and $O_4$ can be arranged, for large $q_i$, in one strip of height $6$.  The number of solid $1$-tiles needed is at most the total area of this strip, which is $6p_i$.  Then, there are up to $3$ copies of the triangular tile $\Delta_2$, which we would like to place so that the unfilled squares are in the corners of the $p_i$ by $q_i$ rectangle.  To do this, we make two strips of height $2$, to go on the bottom and top of the $p_i$ by $q_i$ rectangle, using a total of at most $4p_i$ solid $1$-tiles.

We have described how to use up all the tiles except the solid $1$-tiles in strips of width $p_i$, in such a way that the number of solid $1$-tiles used is at most linear in $q_i$.  Because the total number of solid $1$-tiles is at least $\varepsilon\cdot p_iq_i$, for large $p_i$ we cannot have used all of them.  Thus, we arrange the remaining solid $1$-tiles in another rectangular strip of width $p_i$.  We stack these strips to fill the $p_i$ by $q_i$ rectangle, with the copies of $\Delta_2$ at the corners, to make a rigid cycle tile.
\end{proof}

We are now ready to combine Lemmas~\ref{lem:sum} and~\ref{lem:pack} to prove our main theorem.

\begin{proof}[Proof of Theorem~\ref{thm:strong-lower}]
Because $(x, y)$ is in the interior of the feasible region, either it is in the interior of the segment with vertices $(1, 0)$ and $s_k$, or it is in the interior of the triangle with vertices $(1, 0)$, $s_k$, and $s_{k+1}$.  For sufficiently large $i$, we apply Lemmas~\ref{lem:sum} and~\ref{lem:pack} to construct a rigid cycle tile $T$ with $n(T), j(T), p(T), q(T)$ equal to $n_i, j_i, p_i, q_i$.  The largest clusters in $T$ are those coming from copies of $Y_{2(k+1)}$ if $k \leq 4$ and from copies of $\Delta_{k + 1}$ if $k \geq 4$, so in any case they have less than $3(k+1)^2$ pieces.  Applying Lemma~\ref{lem:permute}, we have
\[\dim H_{j_i}[C(n_i; p_i, q_i)] \geq \frac{n_i!}{\prod_\ell n_{i, \ell}!},\]
with each $n_{i, \ell} \leq 3(k+1)$ and $\sum_\ell n_{i,\ell} = n_i$.  The ratio between $\log$ of the denominator and $\log n_i!$ is
\[\frac{\sum_\ell \log(n_{i, \ell}!)}{\log(n_i!)} \leq \frac{n_i \cdot \log\ (3(k+1)^2)!}{\log\ n_i!} \rightarrow 0,\]
so we have
\[\lim_{i \rightarrow \infty} \frac{\log\ \dim H_{j_i}[C(n_i; p_i, q_i)]}{\log\ n_i!} \geq 1.\]

Combining this lower bound with the upper bound of Lemma~\ref{lem:upper}, we conclude that $\dim H_{j_i}[C(n_i; p_i, q_i)]$ has the factorial growth rate of $n_i!$.
\end{proof}

\section{Open problems}\label{sec:conclusion}

We close with a few suggestions for future directions. \\

\begin{enumerate}
    \item Although some headway is made in \cite{ABKMS20}, we still do not have necessary and sufficient conditions 
    for $H_{j}[C(n; p, q)] \neq 0$. We do not believe that there is any nontrivial homology outside the feasible region $\mathcal{R}$. A slightly weaker necessary condition for nontrivial homology, simpler to state, is the following.

    \begin{conjecture}
    Suppose that $H_j[C(n; p, q)] \neq 0$.  Let $x = \frac{n}{pq}$ and $y = \frac{j}{pq}$.  Then $y \leq \min\{x - \frac{8}{9}x^2, 1-x, \frac{1}{4}\}$.
    \end{conjecture}
    For comparison, it is shown in \cite{ABKMS20} that $y \leq \min\{x , 1-x, \frac{1}{3}\}$.
    
    It seems to us that even if the conjecture holds, this necessary condition is probably not quite sufficient. For example, if $p =q =5$, $n=12$, and $j=6$, then $(x,y)$ is in the feasible region, but we expect that $H_j[C(n;p,q)]=0$. \\

    \item Describe the boundary between the homological liquid and gas regimes. In other words, give necessary and sufficient conditions on $n,j,p,q$ for the natural inclusion map $i:C(n;p,q) \hookrightarrow C(n)$ to induce an isomorphism $i_*: H_{j}[C(n; p, q)] \to H_j [C(n)]$.
    
    \begin{conjecture}
    Suppose that $p, q \ge j+2$. Then 
    \[
    i_*: H_{j}[C(n; p, q)] \to H_j [C(n)]
    \]
    is an isomorphism if and only if $pq - n \ge (j+1)(j+2)$.
    \end{conjecture}
    This conjecture is consistent with the data we computed for $n \le 6$ in \cite{ABKMS20}. Note that $pq-n$ is the amount of ``space'' on the board, i.e.,\ the total number of squares in the rectangle minus the number of sliding pieces. \\
    
    \item Sharper estimates for the Betti numbers would be welcome. It may be that Theorem \ref{thm:strong-lower} can be improved to the following statement capturing the ``second-order'' asymptotics.
    \begin{conjecture} \label{conj:stronger}
Let $(x, y)$ be any point in the interior of $\mathcal{R}$.   Let $n_i, j_i, p_i, q_i$ be sequences such that $p_i \leq q_i$, $p_i \rightarrow \infty$, $\frac{n_i}{p_iq_i} \rightarrow x$, and $\frac{j_i}{p_iq_i} \rightarrow y$.  Then 
\[
\log \dim H_{j_i}[C(n_i; p_i, q_i)] = n_i \log n_i + C_{x,y} n_i + o(n_i).
\]
Here $C_{x,y}$ is a constant that only depends on $x$ and $y$.\\
\end{conjecture}
    \item Estimates for the Betti numbers of configuration spaces of hard disks in a box (rather than hard squares) would be interesting. These configuration spaces were studied earlier, for example in \cite{BBK14} and \cite{Alpert17}, but so far little seems to be known.  For instance, for $n$ unit disks in a $p$ by $q$ rectangle, can we find values of $x = \frac{n}{pq}$ and $y = \frac{j}{pq}$ for which we can prove the conclusion of Theorem~\ref{thm:strong-lower}?
\end{enumerate}

\bibliographystyle{amsalpha}
\bibliography{refs}

\end{document}

%% file: feasible.tex
\begin{figure}
\begin{tikzpicture}[scale = 6]
\fill [opacity = 0.1] (0,0)--(1/3,1/3)--(2/3,1/3)--(1,0)--cycle;
\draw [blue, line width = 0.3mm](0,0)--(1,0)--(3/4,1/4)--( 3/8 , 1/4 )--( 1/4 , 7/36 )--( 3/16 , 5/32 )--( 3/20 , 13/100 )--( 1/8 , 1/9 )--( 3/28 , 19/196 )--( 3/32 , 11/128 )--( 1/12 , 25/324 )--( 3/40 , 7/100 )--( 3/44 , 31/484 )--
( 1/16 , 17/288 )--( 3/52 , 37/676 )--( 3/56 , 5/98 )--( 1/20 , 43/900 )--cycle;
\draw[densely dashed] plot[domain=0:9/8] (\x, {\x-(8/9)*(\x)^2});
\draw [densely dashed] (0,1/4)--(1,1/4);
\draw (0,0)--(0.7,0.7);
\node at (0.75,0.75) {$y=x$};
\draw (1,0)--(0.3,0.7);
\node at (0.25,0.75) {$y=1-x$};
\draw (0,1/3)--(1,1/3);
\node at (1.1,1/3) {$y=1/3$};
\node at (1.1,1/4) {$y=1/4$};
\node at (1.1,1/7) {$y=x-(8/9)x^2$};
\draw [line width =0.5mm,->] (0,0)--(1.25,0);
\draw [line width =0.5mm,->] (0,0)--(0,0.75);
\node at (0.625,-0.04) {$\mathbf{x=n/pq}$};
\node[rotate=90] at (-0.05,0.375) {$\mathbf{y=j/pq}$};

\fill [blue,opacity = 0.2,line width=0.01mm](0,0)--(1,0)--(3/4,1/4)--( 3/8 , 1/4 )--( 1/4 , 7/36 )--( 3/16 , 5/32 )--( 3/20 , 13/100 )--( 1/8 , 1/9 )--( 3/28 , 19/196 )--( 3/32 , 11/128 )--( 1/12 , 25/324 )--( 3/40 , 7/100 )--( 3/44 , 31/484 )--
( 1/16 , 17/288 )--( 3/52 , 37/676 )--( 3/56 , 5/98 )--( 1/20 , 43/900 )--cycle;

\end{tikzpicture}
\label{fig:summary}

\caption{It is proved in \cite{ABKMS20} that all nontrivial homology must lie in the shaded region $0 \le y \le \min\{ x, 1-x, 1/3 \}$. The blue region is what we call the feasible region $\mathcal{R}$.}
\label{fig:main}
\end{figure}